\documentclass[11pt,a4paper,twoside]{article}

\usepackage[top=1in, bottom=1in, left=1.25in, right=1.25in]{geometry}

\usepackage[utf8]{inputenc}
\usepackage[english]{babel}

\usepackage{amsthm}
\usepackage{amssymb}
\usepackage{amsmath}
\usepackage{tikz-cd}

\usepackage{enumitem}

\usepackage{graphics}

\newtheorem{theoremalpha}{Theorem}

\newtheorem{theorem}{Theorem}[section]
\newtheorem*{theorem*}{Theorem}
\newtheorem{proposition}[theorem]{Proposition}
\newtheorem{lemma}[theorem]{Lemma}

\theoremstyle{definition}
\newtheorem{definition}[theorem]{Definition}

\DeclareMathOperator{\reals}{\mathbb{R}}
\DeclareMathOperator{\naturals}{\mathbb{N}}

\DeclareMathOperator{\supp}{supp}
\DeclareMathOperator{\inter}{int}

\newcommand{\dd}{\mathrm{d}}

\newcommand{\fu}{{H}^u}
\newcommand{\fs}{{H}^s}

\newcommand{\fcs}{{W}^{ws}}
\newcommand{\fuc}{\tilde{{H}}^u}
\newcommand{\fsc}{\tilde{{H}}^s}
\newcommand{\fcuc}{\tilde{W}^{wu}}
\newcommand{\fcsc}{\tilde{W}^{ws}}
\newcommand{\cali}{\tilde {I}}
\newcommand{\vc}{\tilde{V}^u}

\newcommand{\tm}{T^1M}

\newcommand{\tuc}{T^1 \tilde  M}

\newcommand{\bound}{\partial \tilde M}
\newcommand{\doublebound}{\partial^2 \tilde M}

\title{Unique ergodicity of the horocycle flow of a higher genus compact surface with no conjugate points and continuous Green bundles}
\author{Sergi Burniol Clotet\\
LPSM, Sorbonne Université, 4 Place Jussieu, 75005 Paris, France\\
(email: sergi.burniol\_clotet@upmc.fr)}
\date{November 10, 2022}

\begin{document}

\setlength{\parskip}{0.5ex plus 0.5ex minus 0.2ex}
\maketitle

\begin{abstract}
    We show that the horocyclic flow of an orientable compact higher genus surface without conjugate points and with continuous Green bundles is uniquely ergodic. The result applies to nonflat nonpositively curved surfaces and generalizes a classical result of Furstenberg and Marcus in negative curvature. The proof relies on the definition of a uniformly expanding parametrization on the quotient by the strips of the surface.
\end{abstract}



\section{Introduction}
It is well known that the horocyclic flow of a compact hyperbolic surface is uniquely ergodic, first established by H. Furstenberg \cite{Furstenberg73}. B. Marcus generalized the result for compact surfaces with variable negative curvature \cite{Marcus75}. The proof relies on the uniform hyperbolicity of the associated geodesic flows. The main result of this article is a generalization for a large class of surfaces with non uniformly hyperbolic geodesic flow.

\begin{theoremalpha}
	Let $M$ be an orientable compact $C^\infty$ surface without conjugate points, with genus higher than one and continuous Green bundles. Let $h_s$ be a horocyclic flow on the unit tangent bundle $T^1 M$ of $M$. Then there is a unique Borel probability measure on $T^1 M$ invariant by the flow $h_s$.
\end{theoremalpha}

Nonflat compact surfaces with nonpositive curvature or without focal points satisfy the hypothesis of the theorem and the conclusion is new for these settings. In particular, it is valid for the famous example of a flat cylinder with two negatively curved compact ends. The article also generalizes some partial results obtained previously by the author in nonpositive curvature \cite{Burniol21,Burniol21b}. 

Stable and unstable horocycles are well defined for surfaces without conjugate points and form continuous minimal foliations by curves of the unit tangent bundle of the surface \cite{Eberlein77}. A continuous parametrization of these curves is called a horocyclic flow. Unique ergodicity is a property which does not depend on the parametrization, hence it is enough to work with a fixed one, for instance, by the arc-length of horocycles. When the curvature of $M$ is strictly negative, the horocycles coincide with the stable and unstable manifolds of the geodesic flow.

The key idea in the proof of Marcus in negative curvature is the definition of a uniformly expanding parametrization of the horocycles, which gives an explicit description of Birkhoff averages on horocycles when they are pushed by the geodesic flow. Then it is enough to justify that horocycles are equidistributed under the action of the geodesic flow, which implies the pointwise convergence of Birkhoff averages towards a constant, and thus unique ergodicity (see \cite{Coudene1} for example).
Beyond the strictly negative curvature case, Marcus's method cannot be applied directly to the horocycle flow on the unit tangent bundle of a surface without conjugate points, because the uniformly expanding parametrization does not necessarily exist. We will need to find a way to get around this problem.

One of the main difficulties when dealing with weakly hyperbolic geodesic flows is the presence of strips, that is, families of geodesics that stay at bounded distances in the past and in the future when viewed in the universal cover of the manifold. They are obstructions to expansiveness of the geodesic flow, and they usually imply a lack of contraction or expansion in the horospheres with vectors tangent to them. On a surface without conjugate points and genus higher than one, since it is a visibility manifold, a strip appears exactly when there is a nontrivial intersection of a stable horocycle with an unstable one \cite{RiffordRuggiero21}.
K. Gelfert and R. Ruggiero dealt with these strips by identifying them, thus obtaining a quotient space with a continuous flow semiconjugated to the geodesic flow. They show that this quotient, when the Green bundles are continuous, is a topological 3-manifold, and that the quotient flow is expansive, topologically mixing and has local product structure \cite{GelfertRuggiero2, GelfertRuggiero}. As an application, they deduce the uniqueness of the measure of maximal entropy of the geodesic flow.

It is on this quotient that a uniformly expanding horocyclic flow will be defined. Coudene's theorem will guarantee that this flow is uniquely ergodic. Although this horocyclic flow cannot be lifted to the unit tangent bundle of the surface, we succeed in proving that any horocyclic flow above needs to be uniquely ergodic also.

The paper is organized as follows: in Section \ref{geometry} we introduce horocycles and the main tools to study them, including the quotient space defined by Gelfert and Ruggiero. Then, in Section \ref{measures}, we resume the study of the Patterson-Sullivan measure started in \cite{ClimenhagaKnieperWar21a} and define the family of expanding measures on the unstable horocycle.
The next step, Section \ref{quotient}, will be to define the parametrization of the horocyclic flow on the quotient and prove the unique ergodicity. Finally, in Section \ref{lift}, we will need to
prove the technical result that enables to lift unique ergodicity.\\

\noindent
\begin{minipage}[c]{0.08 \linewidth}
	\includegraphics{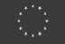}	
\end{minipage}
\begin{minipage}[c]{0.9\linewidth}
	This project has received funding from the European Union’s Horizon 2020 research and innovation program under the Marie Skłodowska-Curie grant agreement No. 754362.
\end{minipage}

\section{Geometry of surfaces without conjugate points}\label{geometry}

Let $M$ be a connected orientable compact $C^\infty$ Riemannian surface of genus strictly higher than one. We denote by $T^1 M$ its unit tangent bundle and by $\pi: T^1 M\to M$ the bundle map. We assume that $M$ has no conjugate points, which means that, at every point $x\in M$, the exponential map $\exp_x: T_pM \to M$ is non-singular. This in turn implies that $\exp_x$ is a universal covering map. We take one universal cover $p: \tilde M \to M$ of $M$ and we equip $\tilde M$ with the pullback metric. Then $\tilde M$ has no conjugate points again and, moreover, its geodesics are globally minimizing.

A unit tangent vector $v\in T^1M$ generates a unit speed geodesic $c_v:\reals \to \tilde M$. The geodesic flow $g_t: T^1M\to T^1M , \,t\in \reals$, is the translation of vectors along geodesics, i.e. for $v\in T^1M$, $g_t(v)=\dot{c}_v(t)$. The same notations will be used for the universal cover $\tilde M$. 

The classical construction of the \textit{boundary} of $\tilde{M}$ allows us to understand better the structure of $\tilde{M}$. Two geodesic rays $\sigma_1,\sigma_2: [0,+\infty)\to \tilde M$ are \textit{asymptotic} if there exists $C>0$ such that for all $t\ge0$, we have $d(\sigma_1(t),\sigma_2(t))\le C$. Since being asymptotic is an equivalence relation, we can consider the set $\partial \tilde M$ of equivalence classes of geodesic rays on $\tilde M$.
The universal cover $\tilde M$ satisfies the following property, called the uniform visibility axiom: for every $\varepsilon>0$, there exists $R>0$ such that for every geodesic segment $\sigma:[a,b]\to \tilde M$ and every $p\in \tilde M$, 
$$d(\sigma([a,b]),p)\ge R \implies \angle_p(\sigma(a),\sigma(b))\le \varepsilon.
$$
This is true because $M$, which is by hypothesis an orientable compact surfaces of genus higher than one, admits a hyperbolic metric \cite[Theorem 5.1]{Eberlein72a}. As a consequence, the boundary of $M$ has a relatively nice structure. For $v\in T^1 \tilde M $, let $v_+$ and $v_-$ denote, respectively, the classes in $\partial \tilde M$ of the positive and the negative rays generated by $v$. Also, for $x,y\in \tilde M,x\not = y$, we denote by $V(x,y)$ the unique vector in $T_x^1 \tilde M $ tangent to the geodesic joining $x$ to $y$.

\begin{theorem}\label{boundary_properties} \cite{Eberlein72a}
Let $M$ be a compact connected surface without conjugate points of genus greater than one.
    \begin{enumerate}
        \item For every $x\in \tilde M$ and $\xi\in \partial \tilde M$, there exists a unique vector $v\in T^1_x \tilde M$ such that $v_+=\xi$. We denote this vector by $V(x,\xi)$. 
        \item There is a topology on $\bar M= \tilde M \cup \partial \tilde M$ which extends that of $\tilde M$, with a basis formed by open sets of $\tilde M$ together with sets of the form 
        $$
        T_{v,\varepsilon, R}:=\{q\in \bar M\,|\, \angle(V(\pi(v),q),v)<\varepsilon \text{ and } d(\pi(v),q)>R \text{ if }q\in\tilde M\}.
        $$
        with $v\in T^1\tilde M$, $\varepsilon,R>0$.
        \item The map $\{(x,y)\in \tilde M \times \bar M \,|\, x\not = y \} \to T^1\tilde M$, $(x,y)\mapsto V(x,y)$ is continuous. Its restriction to $\tilde M \times \partial \tilde  M \to T^1\tilde M$ is a homeomorphism. Moreover, $\bar M$ is topologically a closed disk, $\partial \tilde M$ corresponds to the boundary of the disk and $\tilde M$ to the interior.
        
        \item For every two distinct points $\xi,\eta $ at the boundary $\partial \tilde M $, there exists at least one vector $v\in T^1\tilde M$ such that $v_+=\xi$ and $v_-=\eta$.
    \end{enumerate}
\end{theorem}

Another geometric tool that is defined for surfaces without conjugate points are Busemann functions. For $v\in T^1 \tilde M$ and $t>0$, consider the function $b_{v,t}:\tilde M \to \reals $ defined by $b_{v,t}(x)=d(c_v(t),x)-t$. It is easy to show using the triangular inequality that these functions converge pointwise when $t$ goes to infinity. The Busemann function $b_v:\tilde M \to \reals $ is defined by
$$
b_v(x)=\lim_{t\to +\infty}b_{v,t}(x).
$$
\begin{proposition}\label{propertiesBusemann} \cite{Eschenburg77,KnieperThesis}
Let $M$ be a compact connected surface without conjugate points of genus greater than one.
    \begin{enumerate}
        \item For every $v\in T^1 \tilde M$, $b_v$ is $C^1$ with $L$-Lipschitz derivative, where $L$ only depends on the lower curvature bound of $M$.
        \item The functions $b_v$ depend continuously on $v\in T^1 \tilde M$ in the $C^1$-topology.
        \item For all $v\in T^1 \tilde M$ and $x\in \tilde M$, we have $\nabla b_v(x) = V(x,v_+)$.
    \end{enumerate}
\end{proposition}
The first property is true for any manifold without conjugate points and curvature bounded below. For the other two we need some additional hypothesis, such as the axiom of visibility, which holds in our situation. Property 3 states that the integral curves of $\nabla b_v$ are exactly geodesics asymptotic to $c_v$.
    
Horocycles in $\tilde M$ are defined as level sets of Busemann functions, and they are $C^{1,L}$ submanifolds. We consider their lifts to the unit tangent bundle by normal vectors. More precisely,
$$
\tilde H^s(v) := \{ \,-\nabla b_v (x)\, | \, x\in b_v^{-1}(0)\},
$$
$$
\tilde H^u(v) := \{ \,\nabla b_{-v} (x)\, | \, x\in b_{-v}^{-1}(0)\}
$$  
are called respectively the stable and the unstable horocycle of $v\in T^1\tilde M$. We emphasize that all the vectors in $\tilde H^s(v) $ (resp. $\tilde H^u(v)$) have the same positive (resp. negative) endpoint $v_+$ (resp. $v_-$), which is called the center of the horocycle. Notice that by Proposition \ref{propertiesBusemann}, stable and unstable horocyles form two $g_t$-invariant continuous foliations of $T^1\tilde M$ with Lipschitz leaves. Since Busemann functions are invariant by isometries, so are horospheres, and they pass to the quotient $T^1 M$.

We introduce the notation
$$
\beta_\xi(x,y):=b_{V(y,\xi)}(x),
$$
that we will use later. This quantity is equal to the distance between the horocycles centered at $\xi$ passing through $x$ and $y$. It is called a Busemann cocycle and it depends continuously on $(\xi,x,y)\in \partial\tilde M\times \tilde M\times \tilde M$. Finally, we define the sets
$$ \fcsc(v):=\cup_{t\in \reals} g_t \fsc(v),\, \fcuc(v):=\cup_{t\in \reals} g_t \fuc(v),
$$
which will be called the weak stable and unstable leaves.

Since isometries of $\tilde M$ preserve horocycles, they descend to the quotient $T^1M$. When $M$ is compact, horocyclic foliations of $T^1M$ are already known to be minimal in our context.

\begin{theorem}\cite[Theorem 4.5]{Eberlein77}
    Let $M$ be a compact connected surface without conjugate points of genus greater than one. Then each horocycle $H^s(v)$ or $H^u(v)$ is dense in $T^1M$.
\end{theorem}

As stated in Proposition \ref{boundary_properties}, two distinct points of the boundary are always joined by a geodesic. Contrary to strictly negative curvature, there may be more than one geodesic joining the same points at infinity. Two such geodesics are called biasymptotic, because they stay at bounded distances for all time. As usual we prefer to work with vectors, so for $v\in T^1 \tilde M$ we consider the set $S(v)$ of vectors $w$ whose geodesic $c_w$ is biasymptotic to $c_v$,
$$
S(v):=\{ w\in T^1\tilde M \, | \, \sup_{t\in \reals} d(c_v(t),c_w(t))<+\infty \}.
$$
The following lemma tells us that these biasymptotic geodesics appear in strips.

\begin{lemma}\cite[Lemma 3.1]{RiffordRuggiero21}
    For every $v\in T^1\tilde M$,
    $$
    S(v)= \bigcup_{t\in \reals} g_t \tilde I(v), \text{ where } \tilde I(v)=\tilde H^s(v)\cap \tilde H^u(v).
    $$
    Moreover, $\tilde I(v)$ is the arc of a (possibly trivial) continuous simple curve $c: [a, b]\to T^1\tilde M$. If $v\in T^1 \tilde M$ is the lift of a periodic vector $\bar v$, then $S(v)$ is foliated by lifts of periodic geodesics, which all are in the same homotopy class of $c_{\bar v}$ and which all have the same period.
    
Finally, there exists $Q = Q(M) > 0$ such that the Hausdorff distance between any two biasymptotic geodesics in $T^1\tilde M$ is bounded from above by $Q$.
\end{lemma}

The set $\tilde I(v),\, v\in T^1 \tilde M$, will be called the interval of $v$. It is said to be trivial if it is equal to $\{v\}$. In the quotient, we define for $v\in T^1M$ 
$$
I(v)=dp( \tilde I(\tilde v)) \text{ where } \tilde v \text{ is any lift of } v \text{ to }T^1 \tilde M.
$$

The existence of strips is the most obvious difference between the structure of strict negatively curved surfaces and the structure of the surface $M$ with the present conditions. A way to simplify the dynamics of the geodesic flow on the strips is to identify them into single orbits as follows. 

\begin{definition}
	Let $\sim$ be the equivalence relation on $T^1 M$ defined by
	$$
	v\sim w \iff w\in I(v).
	$$
	The quotient of $T^1M$ by this equivalence relation is denoted by $X=T^1M/\sim$, and quotient map by $\chi: T^1M \to X$. We also define a flow $\phi_t: X\to X$ by putting for $\theta \in X,\, t\in \reals$,
	$$
	\phi_t (\theta):=\chi (g_t(v))\text{, where }v \text{ is any vector in the class }\theta.
	$$
\end{definition}

Observe that the flow $\phi_t$ is well-defined, because $g_t$ preserves the intervals $I(v)$, and continuous. By definition, $\chi $ is a semi-conjugation between $g_t$ and $\phi_t$. Similarly, we define the quotient $\tilde X$ of $T^1 \tilde M$ by the intervals $\tilde I(v)$, and the quotient map $\tilde \chi: T^1\tilde M\to \tilde X$. all these objects are related as expressed in the following diagram.

\[
\begin{tikzcd}
	T^1\tilde M \arrow[d, "d\pi"] \arrow[r, "\tilde \chi"] & \tilde X  \arrow[d] \\
	T^1M  \arrow[r,"\chi"] & X
\end{tikzcd}
\]

Gelfert and Ruggiero obtained strong results about the structure of the quotient space and the dynamics of the quotient flow under a regularity assumption on Green subbundles. These subbundles are the graphs of the stable and the unstable solutions of the Ricatti matrix equation introduced by L. W. Green in \cite{Green58} for manifolds without conjugate points with curvature bounded below. Later, P. Eberlein used these bundles to characterize Anosov geodesic flows \cite{Eberlein73b}. 

\begin{theorem}
    Let $M$ be an $n$-manifold without conjugate points and curvature bounded below. There exist two rank $n-1$ subbundles $G^s,\, G^u$  of $TT^1M$ normal to the direction of the geodesic flow and invariant by its differential.
\end{theorem}

These subbundles are just measurable in general, an example of a compact surface without conjugate points but with discontinuous Green subbundles was built by Ballmann, Brin and Burns \cite{BallmannBrinBurns87a}. However, there are large classes of manifolds with continuous Green subbundles. In increasing order, nonpositive curvature, no focal points, no conjugate points and bounded asymptote all have continuous Green subbundles. In this article, we assume that they are continuous. We define the generalized rank $1$ set as
$$
R_1:=\{ v\in T^1 M \, |\, G^u(v)\not = G^s(v)\},
$$
and the expansive set as
$$
\mathcal{E}:=\{ v\in T^1 M \, |\, I(v) =\{v\}\}.
$$

\begin{theorem} \cite{GelfertRuggiero2}
Let $M$ be a compact connected surface without conjugate points of genus greater than one and with continuous
stable and unstable Green bundles.
    \begin{enumerate}
        \item The families $H^s$ and $H^u$ are continuous foliations of $T^1M $ by $C^1$ curves which are
tangent to the stable and the unstable Green bundles, respectively.
    \item The rank $1$ set $R_1$ is $g_t$-invariant, open, and dense in $T^1M$, and it is contained in the expansive set $\mathcal{E}$.
    \item The quotient space $X$ is a compact topological 3-manifold and the quotient flow $\Psi$ is expansive, topologically mixing, and has a local product structure.

    \end{enumerate}

\end{theorem}
The existence of expansive stable and unstable leaves is one of the key ideas that makes the previous theorem work. We state below the result for later use.

\begin{proposition} \cite[Proposition 3.6]{GelfertRuggiero2} \label{expansive_horocycle}Let $M$ be a compact connected surface without conjugate points of genus greater than one.
    Assume that $v\in R_1$ is forward $g_t$-recurrent, that is, there exists a sequence $t_n\to +\infty $ such that $g_{t_n}(v)\to v$. Then, for every $w\in H^s(v)$, we have 
$$
d(g_{t_n}(v),g_{t_n}(w))\xrightarrow[n\to+\infty]{}0 
$$
Moreover, if the Green bundles are continuous, we have
$$
H^s(v)\subset R_1 \subset \mathcal{E}.
$$
The analogous statement holds for $H^u$ as $t\to -\infty$.
\end{proposition}

\section{A special family of measures on the horocycles}\label{measures}

\subsection{The Patterson-Sullivan measure}
Our starting point is the work of Climenhaga, Knieper and War \cite{ ClimenhagaKnieperWar21a, ClimenhagaKnieperWar21b}, which is an extension of the work of the second author for rank $1$ nonpositively curved manifolds \cite{Knieper98}. They introduce conformal densities $\{\sigma_p\}_{p\in \tilde M}$ on the boundary $\bound$ invariant by a group of covering transformations $\Gamma$ of a compact surface without conjugate points of genus greater or equal than $2$ (and also in higher dimension with two additional conditions). Then, they are able to define a measure $\mu_{BM}$ on the unit tangent bundle $\tuc$ which is proved to be the lift of the unique measure of maximal entropy.

Let $\doublebound:= (\bound\times\bound) \setminus \Delta$, where $\Delta $ denotes the diagonal of $\bound\times\bound$. Given $(\xi,\eta)\in \doublebound$ and $p\in \tilde M$, the Gromov product of $\xi $ and $\eta$ at $p$ is
$$ \langle\xi , \eta \rangle_p := \beta_{\xi}(p,x)+\beta_{\eta}(p,x), $$
where $x$ is any point on any geodesic connecting $\xi $ and $\eta$. It does not depend on the choice of $x$. 
We fix a reference point $0\in \tilde M$ and denote by $\delta$ the critical exponent of $\Gamma$, which is positive and coincides with the topological entropy of $g_t$ on $T^1 M$ \cite{FreireMane82}.

According to \cite[Theorem 5.6]{ClimenhagaKnieperWar21a}, on a compact higher genus surface without conjugate points, the lift $\mu_{BM}$ of the measure of maximal entropy for the geodesic flow $g_t$ on $T^1M$ gives full measure to $\tilde{\mathcal{E}}\subset \tuc$ and in restriction to this set it satisfies
	\begin{equation}
	\dd \mu_{BM}(v)= e^{\delta \langle v_-,v_+ \rangle_0 } \dd t \, \dd\sigma_0(v_-)\, \dd \sigma_0(v_+), \label{eq:BowenMargulis}    
	\end{equation}
	where $\dd t$ stands for the Lebesgue measure on the $g_t$ orbit going from $v_-$ to $v_+$.

We start by proving some additional properties of the Patterson-Sullivan measure $\sigma_0$. Recall that $R_1$ is the subset of $\tm$ formed by the vectors where the Green subspaces are linearly independent. Moreover, when a vector is both in $R_1$ and $g_t$-recurrent, then its horocycles are entirely contained in $R_1$ by Proposition \ref{expansive_horocycle}. Let $Rec_1$ be the set of vectors in $R_1$ which are forward and backward $g_t$-recurrent and let $\tilde{Rec}_1$ be its lift to $T^1\tilde M$. A subscript $_+$ on a subset of $\tuc$ denotes its projection to the boundary and a superscript $^c$ denotes its complement.

\begin{lemma}
	We have $\tilde{Rec}_{1+} \cap ( \tilde{R}_1^c)_+ =\emptyset$.
\end{lemma}
\begin{proof}
	Let $\eta\in \tilde{Rec}_{1+}$ and take $v\in \tilde{Rec}_1$ such that $v_+=\eta$. By Proposition \ref{expansive_horocycle}, we know that $\fsc(v)$ is contained in the rank $1$ set, so $\fcsc(v)$ is also in the rank $1$ set by invariance. Then, since the set of vectors pointing positively to $\eta$ is exactly $\fcsc(v)$, $\eta$ is not the endpoint of a vector in $\tilde R_1^c$.
\end{proof}

Let $\sigma_0$ be a Patterson-Sullivan measure on $\bound$.

\begin{lemma}
	$\tilde{Rec}_{1+} $ has full $\sigma_0$-measure.
\end{lemma}

\begin{proof}
	The Bowen-Margulis measure $\mu_{BM}$ is ergodic and fully supported \cite{ClimenhagaKnieperWar21a}. Since $R_1$ is open and $g_t$-invariant it has full $\mu_{BM}$-measure. By the Poincaré recurrence theorem, $\mu_{BM}$-a.e. vector is forward and backward recurrent, which yields that $Rec_1$ has full measure, and so does $\tilde{Rec}_1$.
	
	Now consider the subset $A:=\{v\in T^1 \tilde M \mid v_- \not \in \tilde{Rec}_{1+} \}$ of $T^1 \tilde M$. The expression of the Bowen-Margulis measure on $\tilde{\mathcal{E}}$ (Equation \ref{eq:BowenMargulis}) implies
	$$
	\mu_{BM}(A)=\int_{(\tilde{Rec}_{1+})^c}\int_{\partial \tilde M } \infty 
	e^{\delta (\xi |\eta)_0} \dd \sigma_0(\eta) \dd \sigma_0(\xi),
	$$
	so $A$ is negligible in $T^1 \tilde M$ if and only if 
 $\tilde{Rec}_{1+}$ has full measure in $\partial \tilde M$. Finally, we can observe that
	$$
	A=\{v\in T^1 \tilde M \mid v_- \not \in \tilde{Rec}_{1+} \} \subset \tilde {Rec}_1^c,
	$$
	which implies that $A$ is actually negligible.
\end{proof}

\subsection{Product structure of the measure of maximal entropy}

For an unstable leaf $H=H^u(v)$, let $P_H:H\to \partial \tilde M\setminus \{v_-\}$ denote the projection to the positive endpoint, which is a continuous surjective map.
For every $\eta \in \bound, \, \eta \not = v_-$, the set $P_H^{-1}(\eta)$ is an equivalence class, that is, a subset of the form $\tilde I(v)$ for $v\in P_H^{-1}(\eta)$. Moreover, we know that whenever $w\in \tilde I(v)$, then $b^v=b^w$ \cite[Lemma 2.10]{RiffordRuggiero21}.
The function $\phi_H:\eta \mapsto \exp (\delta b^{P_H^{-1}(\eta)}(0))$ is thus well defined.

We can now define a measure $\mu_H$ on $H$ by establishing that, for every Borel subset $A\subset H$,
\begin{equation}
	\mu_H(A)=\int_{\bound \setminus\{v_-\}} \mathbf{1}_{A_+}(\eta) \phi_H(\eta) \dd \sigma_0(\eta). \label{eq:meas_horocycles}
\end{equation}
Notice that the set of points in the boundary with more than one preimage by $P_H$ has zero measure, so the expression above does define a $\sigma$-additive measure and the following formula is true almost everywhere:
$$ \dd \mu_H(w)= e^{\delta b_w(0)} \dd \sigma_0(w_+). $$
Similarly we can define measures on the stable horocycles which satisfy
$$ \dd \mu_{\fsc(v)}(w)= e^{\delta b_{-w}(0)} \dd \sigma_0(w_-) .$$
Finally, for each $v\in \tuc$, we define a measure $\nu_v$ on the weak stable leaf $\fcsc(v)$ of $v$ by
$$ 
	\nu_v(A)=\int_{\reals} e^{\delta t} \int_{\fsc(g_t v) } \mathbf{1}_A(w) \dd \mu_{\fsc(g_t v)}(w) \, \dd t.
$$
This is not the usual Margulis measure on $\fcsc(v)$ which is uniformly expanded. In fact, $g_{t*}\nu_v=\nu_v$ and $\nu_{g_t v}=e^{-\delta t}\nu_v$.

\begin{proposition} \label{product_measure}
	For every vector $u\in \tuc$ and every Borel subset $A\subset \tuc \setminus \fcsc(u)$, we have
	$$ \mu_{BM}(A)= \int_{\fuc(u)} \int_{\fcsc(v)} \mathbf{1}_A (w) \, \dd \nu_v(w)\, \dd \mu_{\fuc(u)}(v) .$$
\end{proposition}

\begin{proof}
	For $v\in\tuc$, $\eta \in \bound\setminus (\mathcal{E}^c)_+ $ with $\eta\not =v_+$ and $t\in \reals $, let $w^v_{t,\eta}$ the unique vector in $\fsc(g_tv)$ pointing negatively to $\eta$. Also, for $\xi \in \bound\setminus (\mathcal{E}^c)_+$ let $v_\xi$ be the unique vector lying in $\fuc(u)$ pointing to $\xi$. Finally, let $w^{\xi}_{t,\eta} : = w^{v_\xi}_{t,\eta}$.
	
	We prove the proposition by carrying the following computations. First, we write the double integral in terms of the measure on the boundary. Then we apply the equality
		$$
	\beta_{\xi}(0,\pi(v_{\xi}))=\beta_{\xi}(0,\pi(w^\xi_{t,\eta}))-t
	$$
	to simplify. Finally, we integrate in $t$ the indicator function of $A$ at the point $w^{\xi}_{t,\eta}$, which is exactly the Lebesgue measure on the geodesic $(\eta,\xi)$ of the set $A$. We can restrict all the integrals to vectors in $\mathcal{E}$ or endpoints in $ \bound\setminus (\mathcal{E}^c)_+$ because they are of full measure each time, thus we do not need to worry about vectors contained in strips.

\begin{align*}
	\int_{\fuc(u)} & \int_{\fcsc(v) \cap \mathcal{E}} \mathbf{1}_A (w) \, \dd \nu_v(w)\, \dd \mu_{\fuc(u)}(v) 
	\\
	=&\int_{\fuc(u)} \int_{\reals} \int_{\fsc(g_t v) \cap \mathcal{E}} \mathbf{1}_A(w) e^{\delta t} \dd \mu_{\fsc(g_t v)}(w) \, \dd t \, \dd \mu_{\fuc(u)}(v)
	\\
	=&\int_{\fuc(u)} \int_{\reals} \int_{ \bound \setminus (\mathcal{E}^c)_+} \mathbf{1}_A (w^v_{t,\eta}) e^{\delta t} e^{\delta \beta_{\eta}(0,\pi(w^v_{t,\eta}))} \, \dd \sigma_0(\eta ) \, \dd t \, \dd \mu_{\fuc(u)}(v) 
	\\
	=&\int_{\bound \setminus (\mathcal{E}^c)_+} \int_{\reals} \int_{ \bound \setminus (\mathcal{E}^c)_+} \mathbf{1}_A (w^\xi_{t,\eta}) e^{\delta t} e^{\delta \beta_{\eta}(0,\pi(w^\xi_{t,\eta}))} e^{\delta \beta_{\xi}(0,\pi(v_{\xi}))} \, \dd \sigma_0(\eta )  \, \dd t \, \dd \sigma_0(\xi) 
	\\
	=&\int_{\bound \setminus (\mathcal{E}^c)_+} \int_{\reals} \int_{ \bound \setminus (\mathcal{E}^c)_+} \mathbf{1}_A (w^\xi_{t,\eta})  e^{\delta \langle \eta, \xi \rangle_0} \, \dd \sigma_0(\eta )  \, \dd t \, \dd \sigma_0(\xi) 
	\\
	=&\int_{\bound \setminus (\mathcal{E}^c)_+} \int_{ \bound \setminus (\mathcal{E}^c)_+} e^{\delta \langle \eta, \xi \rangle_0} Leb_{\eta,\xi}(A) \, \dd \sigma_0(\eta ) \, \dd \sigma_0(\xi) =
	\mu_{BM}(A).	
\end{align*}
	
\end{proof}


\subsection{Additional properties of the measures on the horocycles}
It is not hard to check that the family of measures $\{\mu_{H^u(v)}\}_{v\in \tuc}$ is $\Gamma$-invariant and is exponentially expanded by the geodesic flow,
$$
\mu_{g_t H}=e^{\delta t} g_{t*}\mu_H,
$$
because $b_{g_t v}(0)=t+b_{ v}(0)$.

We can also show that the measures $\mu_{\fsc(v) } $ defined by (\ref{eq:meas_horocycles}) vary continuously with $v$.

\begin{proposition}\label{continuity_mF}
		The map
	\begin{equation*}
		\begin{matrix}
			\{(v,w)\in T^1 \tilde M \times T^1\tilde M \,|\,w\in H^u (v)\} & \longrightarrow & \reals \\
			(v,w) & \longmapsto & \mu_{H^u(v)}((v,w))
		\end{matrix}
	\end{equation*}
	is continuous.
\end{proposition}

\begin{proof} We want to show the continuity at $(v,w)$, $w\in H^u(v)$. The map $(v',\eta)\in T^1\tilde M\times \partial\tilde M \mapsto \phi_{H^u(v')}(\eta)$ is continuous, so it is bounded by a constant $C$ if we restrict $v'$ to a small enough neighborhood of $v$ and $\eta $ to a relatively compact neighborhood of $[v_+,w_+]$ in $\partial \tilde M\setminus\{v_-\}$. The difference in measure with and interval $(v',w')$ close to $(v,w)$ is 
	$$
	|\mu_{H^u(v)}((v,w))-\mu_{H^u(v')}((v',w'))|\le$$
	$$\le \int_{[v_+,w_+]} 
	|\phi_{H^u(v)}-\phi_{H^u(v')} |\dd \sigma_0 + C \sigma_0([v_+,w_+]\triangle {[v'_+,w'_+]}).
	$$
	
	The second term clearly goes to $0$ when $(v',w')$ approaches $(v,w)$. In the first one, $\phi_{H^u(v')}$ converges pointwise to $\phi_{H^u(v)}$ when $v'\to v$. Moreover, $\phi_{H^u(v')}|_{[v_+,w_+]}$ are dominated by the constant $C$. So the first integral tends to $0$ by dominated convergence.

\end{proof}

To prove one of the next properties we will need the following lemma. For each $v\in\tuc$, we define $l(v)\in [0,+\infty]$ to be the length of the interval $\pi(\tilde I(v))=\pi(\fsc(v))\cap \pi(\fuc(v))$. 

\begin{lemma}\label{lemma_bound_l}
	The function $l$ is bounded by a constant $R>0$.
\end{lemma}
\begin{proof}
	The function $l$ is upper semi-continuous and $\Gamma$-invariant, so the claim is reduced to showing that $l$ cannot take the value $+\infty$.
	
	The fact that $l(v)$ is infinity implies that we can find vectors $w$ in $ \fuc(v)\cap \fsc(v)$ at arbitrarily large Sasaki distance from $v$. As we now explain, this contradicts the following consequence of Morse's theorem \cite{Morse24}: there is a constant $Q>0$ which bounds the Hausdorff distance between any two biasymptotic geodesics. 
	It is enough to consider $w\in \fuc(v)\cap \fsc(v)$ such that $d(\pi(v),\pi(w))>2Q$. The geodesics $\gamma_v$ and $\gamma_w$ are biasymptotic, but $d(\pi(v),\gamma_w(t))\ge d(\pi(v),\pi(w))-|t|>Q$ for $|t|\le Q$ and $d(\pi(v),\gamma_w(t))\ge |b^{v}(\gamma_w(t))| =|b^{w}(\gamma_w(t))|>Q$ for $|t|>Q$.
	
\end{proof}

Finally, we can show the following properties of the measures on the horocycles.

\begin{proposition}\label{properties_muF}
	For every leaf $H=\tilde H^u(v)$, the measure $\mu_{H}$ satisfies the following:
	\begin{enumerate}
		\item it has no point masses,
		\item it is finite on compact subsets,
		\item for every $w\in H$, $v\not \in \supp \mu_H $ if and only if $w\in \inter \tilde I(v)$,  
		\item it gives infinite measure to half-horocycles. 
	\end{enumerate}
\end{proposition}
\begin{proof}
	1. It is true because $\sigma_0$ has no point masses. The latter is a consequence of the shadow lemma \cite[Proposition 5.4(b)]{ClimenhagaKnieperWar21a}, which says that any point has neighborhoods of arbitrarily small measure. 
	
	2. It is true because $\phi_H$ is a continuous function on $\bound\setminus{\{v_-\}}$, so it is bounded on compact subsets.
	
	3. If $w\in H$ is in the interior of $\tilde I(w)$ with respect to the topology on the leaf $H$, since $\mu_{H}(\inter \tilde I(w))=0$, we see that $w$ is not in the support of $\mu_H$. Conversely, suppose that there is an open interval containing $w$ with zero $\mu_H$-measure. If $u\in U\setminus \tilde I(w)$, then $u\not \in \fs(w)$ and actually $w_+$ and $u_+$ are distinct. A simple computation,
	$$
	\mu_H(U)\ge \int_{[w_+,u_+]} \phi_H(\eta) \dd \sigma_0(\eta)>0,
	$$
	yields a contradiction. So	we have proved that $U\subset \tilde I(w)$ and the statement follows. 
	
	4. Consider the horocyclic flow $h_t$ with the Lebesgue parametrization. Let $R>0$ be the bound obtained in Lemma \ref{lemma_bound_l}.
	Assume that 
	$$
	\mu_H(\{h_t(v)\}_{t\ge0})= \sum_{k=0}^{+\infty}\mu_H([ h_{2Rk}(v),h_{2R(k+1)}(v))) <+\infty.
	$$
	Then $\mu_H([ h_{2Rk}(v),h_{2R(k+1)}(v)))\to 0$ when $k\to+\infty$. Let $w$ be an accumulation point of $h_{2Rk}(v)$. By continuity of the measure we get that $\mu_{\fu(w)} ((w,h_{2R}(w)))=0$. But this means that $(w,h_{2R}(w))\subset I(w)$, so $l(w)\ge 2R$, which contradicts the lemma. 
\end{proof}

	\subsection{Measures on the quotient $X$}

We now turn our attention to the quotient $X$, which is where the next arguments will take place. Recall that $X$ is defined as the quotient of $T^1M $ by an equivalence relation whose classes $[v]$ are the intervals $I(v)$. We can also define the quotient $\tilde X $ of $\tuc $ by the equivalence relation which identifies elements in the same interval $\tilde I(v)$. The group $\Gamma $ acts naturally on $\tilde X$ since $\gamma \tilde I(v) = \tilde I(\gamma v)$, and $X$ can be thought as its quotient space.

	We now bring horocycles and their measures to the quotients. We write 	$$V^u([v])=\chi(\fu(v)),\, \tilde{V}^u([v])=\chi(\fuc(v)) $$
	for the quotient horocycles, and then we push the measures,
	$$
	 \mu_{V^s([v])}=\chi_*\mu_{\fu(v)},\,	\mu_{\tilde{V}^s([v])}=\chi_*\mu_{\fuc(v)}.
	 $$

The curves $V^u(\theta)$ form a continuous foliation of $X$ because the charts used in \cite[Lemma 4.4]{GelfertRuggiero2} to show the topological structure of $X$ are in fact foliated charts of $V^u$. Next we transfer the properties of the previous sections about the measures on the horocycles to the quotient.

\begin{proposition}\label{measures_quotient}
	\begin{enumerate}
		\item For all $\theta \in \tilde X$, $\mu_{\vc(\theta)}$ has no point masses.
		\item For all $\theta \in \tilde X$, $\mu_{\vc(\theta)}$ is finite on compact subsets.
		\item For all $\theta \in \tilde X$, $\supp \mu_{\vc(\theta)}=\vc(\theta)$.
		
		\item For all $\theta \in \tilde X$, $\mu_{\vc(\theta)}$ gives infinite measure to half-horocycles.
		
		\item For all $\theta \in \tilde X$, for all $t\in \reals$, for all $\gamma\in \Gamma$, we have
		
		$$
		\mu_{\vc(\phi_t(\theta))}=e^{\delta t} (\phi_t)_*\mu_{\vc(\theta)}
		\quad \mu_{\vc(\gamma(\theta))}=\gamma_* \mu_{\vc(\theta)}.
		$$
		
		\item The map
		\begin{equation*}
			\begin{matrix}
				\{(\theta ,\beta)\in \tilde X \times \tilde X \,|\,\beta \in \vc (\theta)\} & \longrightarrow & \reals \\
				(\theta,\beta) & \longmapsto & \mu_{\vc(\theta)}((\theta,\beta))
			\end{matrix}
		\end{equation*}
		is continuous.
		
	\end{enumerate}
\end{proposition}

\begin{proof}
	1. We have $\mu_{\vc([v])}([w])=\mu_{\fuc(v)}(\cali(w))=0$, because $\cali(w)$ projects to a single point in the boundary.
	
	2. Since $\chi$ is proper, the preimage of a compact set $K$ in $\vc([v])$ is compact and $\mu_{\vc([v])}(K)=\mu_{\fuc(v)}(\chi^{-1}(K))$ is finite by Proposition \ref{properties_muF}.
	
	3. Let $U$ be an open nonempty subset of $\vc([v])$. Then $\chi^{-1}(U)$ is an open neighborhood of $\cali(w)$ in $\fuc(v)$, where $[w]$ is a point in $U$. By Proposition \ref{properties_muF}, the two ends of the interval $\cali(w)\subset \chi^{-1}(U)$ are in the support of $\mu_{\fuc(v)}$, hence $\mu_{\fuc(v)}(\chi^{-1}(U))=\mu_{\vc([v])}(U)>0$.
	
	4. The positive half-horocycle of $\theta $ is given by $\chi(\{h_t(v)\}_{t\ge0})$, where $v$ is any vector in $\theta$. Then this implies
	$$
	\mu_{\vc(\theta )}(\chi(\{h_t(v)\}_{t\ge0}))=\mu_{\fuc(v)}(\{h_t(v)\}_{t\ge0}\cup \cali (v))=
	\mu_{\fuc(v)}(\{h_t(v)\}_{t\ge0})=+\infty.
	$$
	Analogously, we show that the negative half-leaf $\chi(\{h_t(v)\}_{t\le0})$ has infinite measure.
	
	5. This follows directly from the corresponding properties for $\{\mu_{H^u(v)}\}_{v\in \tuc }$.
	
	6. For every $\theta\in \tilde X$ and every $\beta \in \vc(\theta)$, note that $\mu_{\vc(\theta) }((\theta,\beta))=\mu_{\fuc(v)}((v,w))$ for any $v\in \theta $ and any $w\in \beta$. 
	
	We proceed by contradiction: suppose that the map is not continuous at $(\theta,\beta)$. Then there exists two sequences $(\theta_k)_{k\in \naturals}$ and $(\beta_k)_{k\in \naturals}$ converging to $\theta $ and $\beta$, respectively, with $\beta_k\in \vc(\theta_k)$ and such that 
	
	\begin{equation}
		\forall k\quad |\mu_{\vc(\theta)}((\theta,\beta))-\mu_{\vc(\theta_{k})}((\theta_{k},\beta_{k}))|>\varepsilon>0.
		\label{eq:discontinuity}
	\end{equation}
	
	For each $k$ select any $v_k\in \theta_k$ and $w_k\in \beta _k$. Up to taking a subsequence, we can assume that $v_k$ converges to a vector $v\in \tuc$ and $w_k$ converges to $w\in \tuc$. The continuity of the quotient map implies that $v\in \theta $ and $w\in \beta$. But now, by the observation made at the beginning and the continuity of the family of measures $\{\mu_{\tilde H^u(v)}\}_{v\in \tuc}$ proved in Proposition \ref{continuity_mF}, we obtain
	$$
	\mu_{\vc(\theta_k) }((\theta_k,\beta_k))=\mu_{\fuc(v_k)}((v_k,w_k))\to \mu_{\fuc(v)}((v,w))
	 = \mu_{\vc(\theta) }((\theta,\beta)).
	$$
	This is in clear contradiction with Equation (\ref{eq:discontinuity}).
	
\end{proof}

\section{Unique ergodicity of the horocyclic flow on $X$}\label{quotient}

 Now we define a horocyclic flow on $\tilde X$. First, we orient the horocycles in $\tuc$ thanks to the orientation of the surface, and we also get an orientation of the images of the horospheres in $\tilde X$. For each $\theta \in \tilde X$, define $h_s(\theta)$ as the unique point in the positive (resp. negative) sense of $\vc(\theta)$ such that $\mu_{\vc(\theta)}((\theta, h_s(\beta ) ))=|s|$ if $s$ is a positive (resp. negative) real number.
 
 \begin{proposition}\label{parametrization_flow}
 	The family of maps $h_s: \tilde X \to \tilde X$ is a continuous $\Gamma$-invariant flow which satisfies $\phi_t\circ h_s=  h_{se^{\delta t}}\circ \phi_t$ and whose orbits are the curves $\vc(\theta)$. Moreover, $h_s$ preserves the measure $\mu_{\tilde X}=\chi_*\mu_{BM}$.
 	
 \end{proposition}
 
 \begin{proof}
    Since both ends of $\vc (\theta)$ have infinite measure, the flow $h_s(\theta) $ exists, and it is unique because $\mu_{\vc(\theta)} $ has full support. Clearly, $h_s$ is a flow. Let us prove that it is continuous. Let $(s_k,\theta_k) \to (s,\theta)\in \reals\times \tuc$. We want to prove that $h_{s_k}(\theta_k)$ converges to $h_s(\theta)$. We assume $s\ge 0$, the other case being done analogously. We know that the curves $\vc (\beta)$ depend continuously on $\beta$, so there exist a sequence $\beta_k$ converging to $h_s(\theta)$ such that $\beta_k\in \vc(\theta_k)$. By Proposition \ref{measures_quotient}, we know that $\mu_{\vc(\theta_k)}((\theta_k,\beta_k))$ converges to $\mu_{\vc(\theta))}((\theta,h_s(\theta)))={s}$ when $k$ tends to infinity. We deduce then that the measures of the intervals $(\beta_k,h_{s_k}(\theta_k))$ go to $0$. 
	
	We claim that the distance between $\beta_k$ and $h_{s_k}(\theta_k)$ tends to $0$. Assume, contrary to our claim, that, for some $\varepsilon>0$, there is a subsequence $k_i$ such that the distance $d(\beta_{k_i},h_{s_{k_i}}(\theta_{k_i}))$ is greater than $\varepsilon$. The intervals $(\beta_{k_i},h_{s_{k_i}}(\theta_{k_i}))$ are accumulating to some interval in $\vc(\theta)$ of length at least $\varepsilon$.  
	But since the measure of $(\beta_{k_i},h_{s_{k_i}}(\theta_{k_i}))$ tends to $0$, the limiting interval should have zero measure, which is a contradiction because $\mu_{\vc(\theta_k)}$ has full support in $\vc(\theta)$. Finally, since $\beta_k$ converges to $h_s(\theta)$, the sequence $h_{s_k}(\theta_k)$ also converges to this point.

    The $\Gamma $-invariance and the uniformly expanding property of $h_s$ follow from the corresponding properties for $\mu_{\vc(\theta)}$. Finally, the conditional measures of $\mu_{\tilde X}$ along horocyclic orbits are exactly the $\mu_{\vc(\theta)}$, and since $h_s$ preserves each of these, it preserves the measure $\mu_{\tilde X}$.
    
 \end{proof}

 	In this way, we obtain a uniformly expanding continuous flow on the quotient $X$, which will also be denoted by $h_s$.

\begin{proposition}
	The flow $h_s$ on $X$ is uniquely ergodic.
\end{proposition}
\begin{proof}
	We apply the following theorem due to Coudène:
	\begin{theorem}\cite{Coudene1}
		Let $X$ be a compact metric space, $g_t$ and $h_s$ two continuous flows on $X$ which
		satisfy the relation : $g_t \circ h_ s = h_{s e^{\delta t} }\circ g_t $. Let $\mu$ a Borel probability measure	invariant under both flows, which is absolutely continuous with respect to $W^{ws}$, and with full support. Finally assume that the flow $h_s$ admits a dense orbit.
		Then $h_s$ is uniquely ergodic.
	\end{theorem}
	The flows $\phi_t$ and $h_s$, with the parametrization that we have just given, are continuous, satisfy the relation of the theorem above (Proposition \ref{parametrization_flow}) and preserve the measure $\mu_X=\chi_*\mu_{BM}$. The projections of the weak stable leaves $\fcs $ to $X$ are the local weak stable manifolds of the quotient flow $\phi_t$, as proved in \cite{GelfertRuggiero2}. It is also proved that the sets $V^u$ are the strong unstable submanifolds of $\phi_t$. Moreover, the flow $\phi_t$ has a product structure \cite[Proposition 4.11]{GelfertRuggiero2}. 
	Proposition \ref{product_measure} implies that the measure $\mu_X$ is locally the product of measures on the image of weak stable leaves by the measure on a curve $V^u$. 
	In terms of the theorem, this means that $\mu_X$ is absolutely continuous with respect to $W^{ws}$.
	It also has full support \cite[Theorem 1.1]{ClimenhagaKnieperWar21a}. Finally, the existence of a dense orbit of $h_s$ follows from the minimality of the foliation $\fu$ \cite[Theorem 4.5]{Eberlein77}.
\end{proof}

\section{Lift to the unit tangent bundle $T^1 M$}\label{lift}

So far, we have proved the unique ergodicity of a horocyclic flow on the quotient space $X$. In this final section, we will deduce that the horocyclic flow on the unit tangent bundle of the surface $M$ with any parametrization is also uniquely ergodic. For this, we need a technical result allowing to relate the two flows from the point of view of invariant measures. This is done in Subsection \ref{general_res}, and then in Subsection \ref{lift_really} we apply the result to conclude the proof.

\subsection{A general result on flow invariant measures} \label{general_res}
Let $\phi_t$ be a continuous flow without fixed points on a compact metric space $X$. We study the flow locally with the help of flow boxes.
\begin{definition}\label{flow_box}
	An open subset $U$ of $X$ is called a \emph{flow box} if there exists a closed subset $T$ of $U$, $\varepsilon>0$ and a homeomorphism 
	$$ \Phi : T \times (-\varepsilon,\varepsilon) \longrightarrow U $$
	such that $\Phi(x,s)=\phi_s(x)$ for all $(x,s)\in T \times (-\varepsilon,\varepsilon)$. The subset $T$ is called a \emph{transversal} and we write $U=\phi_{(-\varepsilon,\varepsilon)}(T)$ to express the fact that $U$ is a flow box with transversal $T$.
\end{definition}

The existence of transversals and flow boxes is guaranteed at the neighborhood of each point. We observe that if $U=\phi_{(-\varepsilon,\varepsilon)}(T)$ is a flow box, then for each open subset $S$ of $T$ and $0<\varepsilon'<\varepsilon$, $U'=\phi_{(-\varepsilon',\varepsilon')}(S)$ is also a flow box. Since we can find arbitrarily small flow boxes containing any given point, they form a base of the topology. At some point it will be easier to work with a finite number of flow boxes which form a cover of the space; this is possible thanks to the compactness assumption.

Flow boxes describe the flow locally, but in order to recover the global dynamics of the flow, we will need extra information provided by the holonomies.

\begin{definition}
	A \emph{holonomy} is a map $\Theta : T_1\to T_2$ between two transversals which is a homeomorphism onto its image, and such that, for every $x\in T_1$, $\Theta(x)$ lies in the orbit of $x$ by $\phi_t$.
\end{definition}

We remark that holonomies exist locally, that is, given two transversals $T_1$ and $T_2$, if $x\in T_1$ and $\phi_t(x)\in T_2$ for some $t\in \reals$, then for every $y\in T_1 $ close to $x$ there exists $t_y$, such that $y\mapsto t_y$ is continuous, with $t_x=t$ and such that $\phi_{t_y}(y)\in T_2$. So $\Theta(y)=\phi_{t_y}(y)$ is a holonomy.

Our goal is to study the measures preserved by the flow $\phi_t$ in terms of measures on the transversals of the flow.

\begin{definition}
	A finite Borel measure $\mu$ on $X$ is said to be \emph{invariant} by the flow $\phi_t$ if $\mu(\phi_t(A))=\mu(A)$ for all Borel subset $A$ of $X$. Let $\{\mu_T\}_T$ be a family of finite Borel measures indexed on all possible transversals $T$ to the flow $\phi_t$, with each $\mu_T$ supported on the transversal $T$. The family $\{\mu_T\}_T$ is \emph{invariant under holonomy} if every holonomy map $\Theta : T_1\to T_2$ between two transversals preserves the measures, i.e. $\Theta_*\mu_{T_1}=\mu_{T_2}|_{\Theta(T_1)}$.
\end{definition}

\begin{proposition}\label{correspondence} Let $\phi_t$ be a continuous flow without fixed points on a compact metric space $X$. 
	There is a correspondence between finite Borel measures on $X$ invariant by $\phi_t$ and families $\{\mu_T\}_T$ of finite Borel measures on the transversals invariant under holonomy.
\end{proposition}


We only recall the main idea, the details can be found in \cite[§ 2]{BeboutoffStepanoff}. If $\mu$ is invariant by $\phi_t$, the measure $\mu_T$ of a Borel subset $A $ of $T$ can be defined as 
$$\mu_T(A)=\mu(\phi_{(-\varepsilon,\varepsilon)}(A))/2\varepsilon$$
where $\varepsilon$ is chosen small enough such that $\phi_{(-\varepsilon,\varepsilon)}(T)$ is a flow box. Conversely, the measure $\mu$ is defined from $\mu_T$ on a flow box $U=\phi_{(-\varepsilon,\varepsilon)}(T)$ as $\Phi_*(\mu_T \otimes Leb)$ where $\Phi $ is the homeomorphism of Definition \ref{flow_box} and $Leb$ is the Lebesgue measure on the interval $(-\varepsilon,\varepsilon)$. One can check that these definitions do not depend on the choices made and that they glue up together, that the measures thus obtained satisfy the invariance properties, and that one construction is the inverse of the other.

We want to represent an invariant measure $\mu$ by their measures on the transversals $\mu_T$. It is clear that a lot of information is redundant. For example, it would be enough to take the transversals associated to a finite cover by flow boxes. Then the measure on the rest of the transversals is recovered as push forwards from measures on the finite collection of transversals. In fact, we need a bit less.

\begin{definition}\label{complete_set_sections}
	A finite set of transversals $T_1,\dots, T_n$ will be called \emph{complete} if every point in $X$ is in the orbit of an element of one of the $T_i$'s.
\end{definition}

\begin{proposition}\label{complete_is_enough}
	An invariant measure $\mu$ is determined by the measures $\mu_{T_1},\dots ,\mu_{T_n}$ on a complete set of transversals $T_1,\dots,T_n$.
\end{proposition}

\begin{proof}
	In view of Proposition \ref{correspondence}, we only need to show that the measures $\mu_{T_1},\dots ,\mu_{T_n}$ determine a family $\{\mu_T\}_T$ of measures on the transversals invariant under holonomy. We consider any transversal $T$ and a point $y\in T$. Since the family of transversals is complete $y=\phi_t(x)$ for some $x\in T_i$. Now consider the holonomy $\Theta_{T_i,S}$ from a small neighborhood of $x$ in $T_i$ to a neighborhood $S$ of $y$ in $T$. The only way to define a measure on $S$ that will be preserved by holonomy is to push forward the measure $\mu_{T_i}$ by the holonomy $\Theta_{T_i,S}$. The definition does not depend on the choice of $T_i$ and $x$, because if we have a different holonomy $\Theta_{T_j,S}$ from a neighborhood in $T_j$ to $S$, then $\Theta_{T_i,S*}\mu_{U_i}=(\Theta_{T_i,T_j}\circ \Theta_{T_j,S})_*\mu_{U_i}=\Theta_{T_j,S*}(\Theta_{T_i,T_j*}\mu_{U_i})=\Theta_{T_j,S*}\mu_{U_j}$ thanks to the fact that the measures on $T_i$ are preserved by holonomy. Moreover if we take another neighborhood $S'$ of a possibly different point $y'\in T$, then the definitions of the measure coincide on $S\cap S'$ by a similar argument.
	
	Consequently, we have a well defined family of measures $\{\mu_T\}_T$ on the transversals. It remains to show that they are invariant under holonomy. This immediately follows from the fact that the measures $\{\mu_{T_i}\}_{1\le i\le n}$ already have this property.
\end{proof}

We would like to remark that if the flow $\phi_t$ is minimal (that is, every orbit is dense in $X$), then any transversal $T$ of $\phi_t$ is complete. Indeed, the orbit of any point $x\in X$ passes through an open flow box $U$ around $T$, so the orbit intersects $T$ and $x$ is in the orbit of a point of $T$.

Finally we get at the main point of this section. We have seen that an invariant measure is determined by a finite number of transversal measures. So if we have two flows on two possibly different spaces, each with their transversals, but the dynamics on these transversals are essentially equal, then both flows will have the same invariant measures.

\begin{definition}
	Let $X$, $Y$ be two compact metric spaces and let $\phi_t:X\to X$, $\psi_t:Y\to Y$ be two continuous flows. Let $F: T\subset X\to Y$ a map defined on a subset $T$ of $X$. We say that the map $F$ preserves the holonomy if for all $x,y\in T$, $x$ and $y$ lie in the same orbit of $\phi_t$ if and only if $F(x)$ and $F(y)$ lie in the same orbit of $\psi_t$.
\end{definition}

\begin{theorem}\label{theorem_transversals}
	Let $X$, $Y$ be two compact metric spaces and let $\phi_t:X\to X$, $\psi_t:Y\to Y$ be two continuous flows without fixed points. Let $(T_i)_{1\le i\le n}$, $(T_i')_{1\le i\le n}$ be complete sets of transversals for the flows $\phi_t$, $\psi_t$, respectively. Assume that there is a map $F:\sqcup_{i=1}^n T_i \to \sqcup_{i=1}^n T_i'$ which preserves the holonomy and such that for all $i\in \{1,\dots , n\}$, $F|_{T_i}:T_i\to F(T_i)=T_i'$ is a homeomorphism.
	Then there is a correspondence between the finite Borel measures on $X$ invariant by $\phi_t$ and those on $Y$ invariant by $\psi_t$.
	
	In particular, $\phi_t$ is uniquely ergodic if and only if $\psi_t$ is uniquely ergodic.
\end{theorem}

\begin{proof}
	Since invariant measures can be thought as families of measures $\{\mu_{T_i}\}_i$ and $\{\mu_{T_i'}\}_i$ on the complete sets of transversals invariant by holonomy, we only have to prove the correspondence between the latter objects. This correspondence is induced by the map $F$. Given a family $\{\mu_{T_i}\}_i$ of measures invariant under holonomy, we take the pushforward by $F$ to obtain a family of measures $\{F_* \mu_{T_i}\}_i$ on the transversals $T_i'$. Let us check that they are also invariant under holonomy. Consider a holonomy $\Theta : U\subset T_i'\to T_j'$. Since $\Theta$ preserves the holonomy, the map $F^{-1}\circ \Theta \circ F:F^{-1}(U)\subset T_i \to T_j$ is a holonomy again. By the invariance of the first family of measures we have 
	$$(F^{-1}\circ \Theta \circ F)_*\mu_{T_i}=\mu_{T_j}|_{(F^{-1}\circ \Theta)(U)}.$$
	Taking $F_*$ we obtain
	$$ \Theta_* (F_*\mu_{T_i})=(F_*\mu_{T_j})|_{ \Theta(U)},$$
	which proves the invariance of the $F_*\mu_{T_i}$.
	This process is reversible because $\Theta$ is inversible, so we have a bijection.
\end{proof}
	
	The result applies when one flow is a change of time of the other and also when they are conjugated. But neither of these situations is the one where we will apply the theorem. There is no obvious relation between the horocyclic flows on the unit tangent bundle $T^1M$ and on the quotient space $X$, in fact we do not even have a homeomorphism between both spaces.

\subsection{Transversals to the horocyclic flows}\label{lift_really}

Our final goal is to show the unique ergodicity of the horocyclic flow on $T^1M$. For this, we choose a continuous parametrization of the horocyclic foliation $\fu$, for example the parametrization $h^L_t$ by the arc length of horocycles. We want to apply the result of the previous section to the flows $h^L_t$ on $T^1M$ and $h_t$ on $X$. The natural projection $\chi:T^1M\to X$ is not well behaved with respect to these flows, it collapses segments and does no preserve the parametrizations. We will find complete sets of transversals for both flows such that the restriction of $\chi$ on these transversals satisfies the properties of Theorem \ref{theorem_transversals}.

\begin{proposition}
	There exists a transversal $T$ of the flow $h^L_t$ on $T^1 M$ such that $\chi(T)$ is a transversal of the flow $h_t$ on $X$ and $\chi: T\to \chi(T)$ is a homeomorphism which preserves the holonomy.
\end{proposition}

\begin{proof}
	The subset of generalized rank $1$ vectors
	$$R_1=\{v\in T^1 M | \,  G^u(v)\not = G^s(v)\}$$
	of $T^1M$ is nonempty and open because we are assuming the continuity of the Green bundles. Moreover the tangent space to the stable (resp. unstable) leaf $\fs(v)$ is the stable (resp. unstable) Green subspace $G^s(v)$ at each point $v\in T^1 M$. As a consequence, the weak stable leaf $\fcs(v)$ is transverse to the unstable leaf $\fu(v)$ at points $v\in R_1$. Fix $v\in R_1$, and consider a relatively compact neighborhood $T$ of $v$ in $\fcs(v)\cap R_1$, so the unstable leaves are still transverse to $\fcs(v)$ at each point of $T$.
	
	The set $R_1$ does not contain vectors with nontrivial strips, so $\chi: R_1 \to \chi(R_1)$ is a homeomorphism as well as $\chi:T\to \chi(T)$. Since $T$ is a transversal of the horocyclic flow $h^L_t$ and $\bar{T}\subset R_1$, we can consider a flow box of the form $U=h^L_{(-\delta ,\delta)}(T)$ included in $R_1$. Now, $\chi(T)$ is relatively compact and with closure included in the open subset $\chi(U)$. There exists $\varepsilon>0$ such that $h_{(-\varepsilon,\varepsilon)}(\chi(T))$ is included in $\chi(U)$. It is straightforward to see that $h_{(-\varepsilon,\varepsilon)}(\chi(T))$ is a flow box with transversal $\chi(T)$ using the continuity of $h_t$.
	
	Finally, $\chi$ automatically preserves the holonomy because for $v,w\in T$, we have $w\in \fu(v)$ if and only if $\chi(w)\in V^u(\chi(v))=\chi(H^u(v))$. 
\end{proof}

We recall that the flow $h^L_t$ is minimal \cite[Theorem 4.5]{Eberlein77}. Since minimality is a property of the orbits of the flows, the unstable leaves, and the map $\chi$ is continuous, then it passes to the images of the unstable leaves, which are exactly the orbits of $h_t$. In short, the flow $h_t$ on $X$ is also minimal. As we observed after Proposition \ref{complete_is_enough}, in the case of minimal flows any transversal is complete in the sense of Definition \ref{complete_set_sections}, so we can directly apply Theorem \ref{theorem_transversals} to the transversals $T$ and $\chi(T)$ of the previous proposition and we get the correspondence between $h^L_t$-invariant measures and $h_t$-invariant measures. Since the flow $h_t$ is uniquely ergodic, we get the desired result.

\begin{theorem}
	The horocyclic flow $h^L_t$ on the unit tangent bundle of a compact surface of genus equal or higher than $2$ without conjugate points and with continuous Green bundles is uniquely ergodic.
\end{theorem}

\bibliographystyle{plain}
\bibliography{general_library}
\end{document}